\documentclass[12pt]{article}
\usepackage{amssymb,amsthm,hyperref,amsmath,bm,amsfonts}
\usepackage{arydshln}
\usepackage{verbatim}
\usepackage{graphicx}
\usepackage{float}

\usepackage{subfigure}
\usepackage{multirow}
\usepackage{color}
\usepackage{appendix}

\usepackage{latexsym,amsmath,amsfonts,amscd}
\usepackage{epsfig}
\usepackage{changebar}
\usepackage{pstricks}
\usepackage{pst-plot}
\usepackage{multirow}
\usepackage{subfigure}
\usepackage{placeins}
\usepackage{amssymb}
\usepackage{enumerate}

\newtheorem{theorem}{Theorem}[section]
\newtheorem{lemma}[theorem]{Lemma}

\newtheorem{remark}{Remark}[section]
\numberwithin{equation}{section}

\title{Turnpike property for hierarchical optimal control problems: from particle systems to hydrodynamic equations}
\author{
Michael Herty \thanks{Chair in Numerical Analysis, IGPM, RWTH Aachen University, Templergraben, 55, D-52062 Aachen, Germany (herty@igpm.rwth-aachen.de)}\qquad
Yizhou Zhou\thanks{Corresponding author, IGPM, RWTH Aachen University, Templergraben, 55, D-52062 Aachen, Germany (zhou@igpm.rwth-aachen.de)} 
}

\date{\today}
\begin{document}
\maketitle{}

\begin{abstract}
This work is concerned with a hierarchical framework of optimal control problems connecting interacting particle systems, the mean field limit equations, and associated hydrodynamic models. By assuming the existence of solutions, we establish the exponential turnpike property for each level of the hierarchy, showing that optimal trajectories remain close to the associated steady states over long time horizons. The results demonstrate that the exponential turnpike behavior persists consistently across scales, providing a unified connection between microscopic, kinetic, and macroscopic optimal control frameworks.
\end{abstract}

\hspace{-0.5cm}\textbf{Keywords:}
{\small Exponential turnpike property, particle system, mean-field limit, hydrodynamic limit, optimal control \\}

\hspace{-0.5cm}\textbf{AMS subject classification:} {\small 93C20, 35Q89, 49N10}

\section{Introduction}

Our work investigates the turnpike property in the context of controlled, second-order interacting particle systems and their macroscopic limits. Specifically, we consider time-dependent systems where each particle obeys a controlled ordinary differential equation (ODE), leading to an ensemble of $N$-coupled differential equations. As the number of particles $N$ tends to infinity, the collective behavior is formally described by a mean-field limit equation. For second-order dynamics, this mean-field model can be further approximated by hydrodynamic or fluid-dynamic equations obtained via moment closure. We aim to understand how the turnpike property persists across these different levels of description.

Interacting particle systems have been extensively studied as multiscale models for complex collective dynamics. Such systems appear in various contexts, including swarming and collective motion \cite{DM08a,MT14a,CDP09,AP13,Alb+19,CKPP}, crowd dynamics \cite{ABCK}, traffic management \cite{TZ},  opinion dynamics \cite{AHP}, machine learning frameworks \cite{GKYP24-1,GKYP24-2,RZ23,MMN18,SS20,BT19} and so on.
Recent developments further emphasize the importance of control and learning mechanisms in such systems, leading to the concept of active particles, which connect collective dynamics with control and optimization. Comprehensive accounts of these topics can be found in \cite{MR3642940,MR3969953}. Control problems at the microscopic and mean-field levels have been established in works e.g. \cite{ACFK,CFPR-13,FS,BeBr19,BPTT21,HeRi-CMS19,PRT-2015}. For uncontrolled systems, the hierarchical derivation from microscopic dynamics to kinetic and fluid-dynamic models is well understood \cite{DM08a,FHT-11,HaTadmor,DHL-MMS14,CCH-2014}. However, the extension of control theory to the fluid-dynamic level remains largely unexplored.

We begin with a particle description given by a Cucker-Smale type equation \cite{HaTadmor}
\begin{eqnarray}\label{intro-particle-eq}
\begin{aligned}
    \frac{dx_i}{dt} =  v_i, \qquad
    \frac{dv_i}{dt} =  \frac{1}{N}\sum_{j=1}^N \Psi(x_i,x_j)(v_j-v_i) + f_i(t)
\end{aligned}
\end{eqnarray}
with $i=1,2,...,N$. Here $x_i=x_i(t)\in\mathbb{R}^D$ and $v_i=v_i(t)\in \mathbb{R}^D$ are the spatial position and the velocity of the $i$-th particle, $f_i$ represents the control action. By considering the objective functional
\begin{eqnarray}\label{intro-particle:cost}
\begin{aligned}
    \frac{1}{N}\sum_{i=1}^N\int_{0}^T |v_i(t)-\bar{v}|^2 + \lambda |f_i(t)|^2 dt
\end{aligned}
\end{eqnarray}
with $\lambda$ a weighting parameter and $\bar{v}$ a constant vector, the optimal control problem for the particle systems is formulated. As $N\rightarrow \infty$, the corresponding optimal control problems at the mean field level were well-established in \cite{FS}. Considering the empirical measure 
$\mu_{N}(t,x,v) = \frac{1}{N}\sum_{i=1}^N \delta\left(x-x_i(t)\right)\otimes \delta\left(v-v_i(t)\right)$
with $(x_i,v_i)(t)$ being the solution of \eqref{intro-particle-eq}, the authors in \cite{FS} prove that $\mu_{N}$ converges to the solution of the associated mean field problem
\begin{align}
    &\min_{f\in\mathcal{F}}   \int_{0}^T \int_{\mathbb{R}^{D}}\int_{\mathbb{R}^{D}} |v-\bar{v}|^2d\mu(t,x,v)  + \lambda |f|^2d\mu(t,x,v) ~dx dt, \label{intro-mean-field-pb1} \\[2mm]
    &\partial_t \mu + v\cdot \nabla_x \mu + \nabla_v \cdot ( Q+ f\mu) = 0\label{intro-mean-field-pb2}
\end{align}
with 
$Q = \int_{\mathbb{R}^{D}}\int_{\mathbb{R}^{D}} \Psi(x,y)(v_*-v)\mu(t,x,v)\mu(t,y,v_*)dv_*dy
$
and give the existence of optimal control for both the $N$-particle system and mean field problem.

In the present work, we formally formulate the optimal control problem at the hydrodynamic level. By taking moments in equation \eqref{intro-mean-field-pb2} up to the second order, we formally derive an unclosed hydrodynamic system. To obtain a closed system, we introduce two types of closure ansatz (details can be found in Section \ref{section2.3}). In contrast to the derivation in~\cite{HaTadmor}, our approach additionally accounts for the control terms and the objective functionals. In particular, we observe that the formulation of the objective functional in the optimal control problem is not unique. To address this, we propose a modified objective functional defined according to the upper bound of the cost function in the mean-field problem.

Having hierarchical optimal control
problems from particle systems to hydrodynamic
equations, we further investigate the exponential turnpike property for these problems, which indicates that the optimal solutions remain exponentially close to reference solutions. Here the reference solutions are taken as the optimal solutions to the corresponding static problems. Originally proposed for discrete-time optimal control problems \cite{DSS,Sa}, the turnpike concept has been developed in many directions, leading to a broad range of related results \cite{MR3217211, MR4493557, GeZua2022, receding, MR4402854, MR3470445, MR4079009, use-2, interior-original, main, MR3124890, MR3780457, MR3271298}, to name but a few. 
For first-order systems, it has been shown \cite{main,Herty_Zhou_2025} that the turnpike property prevails under the large particle limit. Therefore, we mainly aim to extend the derived turnpike property for the hydrodynamic limit problems. 

Our main approach is to construct a feedback control that ensures stabilization of the system. Based on this, we establish the cheap control property \cite{main,Herty_Zhou_2025} as well as the exponential turnpike property for the corresponding optimal control problems. All estimates for the particle system are uniform with respect to the number of particles $N$. Thus all results are also expected in the mean-field level as $N\rightarrow \infty$. For both the particle and mean-field settings, the overall procedure closely follows our previous work \cite{Herty_Zhou_2025} on first-order particle systems. For the hydrodynamic problem, we make a strong assumption that the optimal control problem admits a solution with enough regularity. Under this assumption, we design suitable feedback laws and establish both the cheap control property and the exponential turnpike property. According to these results, we observe that the exponential turnpike property persists throughout the hierarchy of optimal control problems, from microscopic level to macroscopic level.

The paper is organized as follows. In Section \ref{Section2}, we introduce the optimal control problems and state basic assumptions for the particle system, the mean-field system, and the hydrodynamic models in Subsections \ref{Section2.1}–\ref{section2.3}, respectively. Sections \ref{Section3} and \ref{Section4} are devoted to the analysis of these problems at all three levels. In Section \ref{Section3}, Subsections \ref{Section3.1}–\ref{Section3.3} address the cheap control property for the three levels from the microscopic to the macroscopic scale. In Section \ref{Section4}, we establish the exponential turnpike property. The main result is presented in Theorem \ref{thm14}. Finally, in Section \ref{Section5}, we illustrate the feedback control laws for the particle systems and hydrodynamic models through numerical experiments.

\section{Hierarchy of control problems}\label{Section2}

In this section, we introduce the hierarchy of optimal control problems and basic assumptions. In Section \ref{Section2.1} and \ref{Section2.2}, we present the optimal control problems for the particle system and its mean field limit. Then in Section \ref{section2.3} we formally derive the corresponding hydrodynamic limit and formulate the optimal control problem in the macroscopic level.

\subsection{Optimal control problem for the particle system}\label{Section2.1}
Consider an $N$-particle interacting system consisting of $i=1,2,\dots,N$ identical particles:
\begin{eqnarray}\label{particle-eq}
\begin{aligned}
    &\frac{dx_i}{dt} =  v_i, \\[2mm]
    &\frac{dv_i}{dt} =  \frac{1}{N}\sum_{j=1}^N \Psi(x_i,x_j)(v_j-v_i) + f_i(t), \\[2mm]
    &x_i(0) = x_{i0},\quad v_i(0) = v_{i0}.
\end{aligned}
\end{eqnarray}
Here $x_i=x_i(t)\in\mathbb{R}^D$ and $v_i=v_i(t)\in \mathbb{R}^D$ represent the spatial position and the velocity of the $i$-th particle. The initial data $(x_{i0}.v_{i0})$ are i.i.d. sampled from a given probability distribution $\mu_{0}(x,v)\in P_2(\mathbb{R}^{2D})$. The interaction kernel $\Psi$ is assumed to be a bi-symmetric and bounded function, i.e., 
\begin{align}
    \Psi(x,y)=\Psi(y,x),\qquad |\Psi(x,y)|\leq C_\Psi
\end{align}
for any $x,y\in\mathbb{R}^D$ and $C_\Psi$ being a positive constant. A prototype example is given by \cite{HaTadmor}
$$
\Psi(x,y)=\frac{C_\Psi}{(1 + |x - y|^2)^\gamma}
$$
for some $\gamma>0$. 
Moreover, in \eqref{particle-eq} $f_i(t)=f(t,x_i(t),v_i(t))\in \mathbb{R}^D$ is the control term. Motivated by \cite{FS,Herty_Zhou_2025}, we take the admissible controls in a set $\mathcal{F}$ given by (i) $f:[0, T] \times \mathbb{R}^{2D} \rightarrow \mathbb{R}^D$ is a Carathéodory function; (ii) $f(t,\cdot)\in W^{1,\infty}_{loc}(\mathbb{R}^{2D})$ for almost every $t \in [0, T]$; (iii) $|f(t,0)| + \|f(t,\cdot)\|_{Lip} \leq C_B$ for almost every $t \in [0, T]$ with $0<C_B<\infty$ a constant.

In what follows, we use the notation for any $v_i=(v_i^{(1)},v_i^{(2)},\dots,v_i^{(D)})\in\mathbb{R}^D$ and $v_i=(v_j^{(1)},v_j^{(2)},\dots,v_j^{(D)})\in\mathbb{R}^D$
$$
\langle v_i, v_j\rangle=\sum_{k=1}^dv_i^{(k)}v_j^{(k)},\qquad |v_i|^2 = \langle v_i, v_i\rangle. 
$$
For $N$ particles, we write $v(t)=(v_1(t),v_2(t),\dots,v_N(t))$ and denote 
$$
\|v\|_N^2 = \frac{1}{N}\sum_{i=1}^N |v_i|^2. 
$$
Similarly, we also express the position and the control by 
$$
x(t)=(x_1(t),x_2(t),\dots,x_N(t)),\qquad 
f(t)=(f_1(t),f_2(t),\dots,f_N(t)).
$$ 

Now we formulate the optimal control problem $\mathcal{Q}_N(0,T,x_0,v_0)$ for the particle systems \eqref{particle-eq} with the following cost function.
\begin{eqnarray}\label{particle:cost}
\begin{aligned}
    \int_{0}^T \|v(t)-\bar{v}\|_N^2 + \lambda \|f(t)\|_N^2 ~dt=\frac{1}{N}\sum_{i=1}^N\int_{0}^T |v_i(t)-\bar{v}|^2 + \lambda |f_i(t)|^2 dt.
\end{aligned}
\end{eqnarray}    
Here $\lambda$ is a weighting parameter
for the control. Notice that only the velocity term is considered in the cost function while the positions for  $x_i(t)$ are not included. We denote $\mathcal{Q}_N(0,T,x_0,v_0)$ as the optimal problem on the time interval $[0,T]$ with initial data $x_0=(x_{10},x_{20},\dots,x_{N0})$, $v_0=(v_{10},v_{20},\dots,v_{N0})$ and $\mathcal{V}_N(0,T,x_0,v_0)$ the minimum value of cost function, i.e. 
\begin{align*}
&\mathcal{V}_N(0,T,x_0,v_0)=\min_{f\in\mathcal{F}}\int_{0}^T \|v(t)-\bar{v}\|_N^2 + \lambda \|f(t)\|_N^2 ~dt\\[1mm]
& \text{subject to equations \eqref{particle-eq}}.
\end{align*}

\subsection{Optimal control problem for the mean field equation}\label{Section2.2}

If the number of particles $N$ in \eqref{particle-eq}
is large enough, we may derive the equation for the probability distribution function $\mu=\mu(t,x,v)$. Formally, by using the BBGKY hierarchy and the molecular chaos assumption, we derive the mean-field limit equation
\begin{align*}
    \partial_t \mu + v\cdot \nabla_x \mu + \nabla_v \cdot ( Q+ f\mu) = 0.
\end{align*}
Here, the collision term is given by
\begin{align*}
Q &= \int_{\mathbb{R}^{D}}\int_{\mathbb{R}^{D}} \Psi(x,y)(v_*-v)\mu(t,x,v)\mu(t,y,v_*)dv_*dy.
\end{align*}
For a detailed derivation, we refer to \cite{HaTadmor} for the case without control term $f$. We can also formally take the mean field limit for the target function in \eqref{particle:cost} and formulate the optimal control problem $\mathcal{Q}_M(0,T,\mu_0)$ as
\begin{eqnarray}\label{mean-field-pb}
\begin{aligned}
    \mathcal{V}_M(0,T,\mu_0)
    =&~\min_{f\in\mathcal{F}}   \int_{0}^T \int_{\mathbb{R}^{D}}\int_{\mathbb{R}^{D}} |v-\bar{v}|^2d\mu(t,x,v)  + \lambda |f|^2d\mu(t,x,v) ~dx dt, \\[2mm]
    &\partial_t \mu + v\cdot \nabla_x \mu + \nabla_v \cdot ( Q+ f\mu) = 0, \\[2mm]
    &\mu(0,x,v) = \mu_0(x,v).
\end{aligned}
\end{eqnarray}
Note that this problem is connected with the problem $\mathcal{Q}_N(0,T,x_0,v_0)$ for the particle system if we consider the empirical measure on $[0,T] \times \mathbb{R}^D$
\begin{equation}\label{empirical}
\mu_{N}(t,x,v) = \frac{1}{N}\sum_{i=1}^N \delta\left(x-x_i(t)\right)\otimes \delta\left(v-v_i(t)\right)
\end{equation}
with $(x_i,v_i)(t)$ $(i=1,2,...,N)$ being the solution of \eqref{particle-eq}.

The existence and uniqueness of the solution $\mu$ to the problem \eqref{mean-field-pb} has been established in \cite{FS}. To recall the theorem,  the definition of the $p-$Wasserstein distance between two probability measures $\mu$ and $\nu$ is given:
$$
\mathcal{W}_p(\mu,\nu) = \inf_{r \in \Gamma(\mu,\nu)} \left( \int_{\mathbb{R}^{2D}} |x-y|^p dr(x,y) \right)^{1/p}.
$$
Here, $\Gamma(\mu,\nu)$ denotes the set of transport plans, i.e., collection of all probability measures with marginals $\mu$ and $\nu$, see also \cite{MR2459454}.

\begin{theorem}\label{thm21}
Assume that the initial data $\mu_0\in P(\mathbb{R}^D)$ in \eqref{mean-field-pb} is compactly supported, i.e., there exists $R > 0$ such that
$\text{supp}~ \mu_0 \subset B(0, R)\subset \mathbb{R}^{2D}$. Moreover, the empirical measure
$\mu_N(0,x,v)$ converges to $\mu_0$ in $\mathcal{W}_1$ distance. Then, there exists an optimal control $f(t,x)$ and a weak equi-compactly supported solution $\mu(t,x,v)$ to the problem \eqref{mean-field-pb}. Namely, for all $t\in [0,T]$ the distribution $\mu(t,x,v) \in C([0, T];P_1(\mathbb{R}^{2D}))$ satisfies $\text{supp}~ \mu(t,\cdot) \subset B(0, R)$ and 
\begin{eqnarray}\label{weak}
\begin{aligned}
    &\int \phi(t,x,v)d\mu(t,x,v) -\int \phi(0,x,v)d\mu_0(x,v) \\[2mm]
    = &\int_{0}^{t} \int \Big[\partial_t \phi + \nabla_x \phi \cdot v + \nabla_v \phi \cdot (Q + f) \Big]d\mu(s,x,v)ds, 
\end{aligned}
\end{eqnarray}
for any $\phi\in C_0^{\infty}([0,T]\times \mathbb{R}^{2D})$.
The optimal solution satisfies
$$
\lim_{k\rightarrow \infty}\mathcal{W}_1(\mu_{N_k}(t,\cdot), \mu(t,\cdot)) = 0
$$
uniformly with respect to $t \in [0, T]$ and $f_{N_k}$ converges to $f$ in $\mathcal{F}$. Here $\mu_{N_k}$ is given by \eqref{empirical} and $(x_i(t),v_i(t),f_{N_k}(t,x))$ is the optimal solution to \eqref{particle-eq} with $N_k$ particles. Moreover, the cost function has the lower semi-continuous property:
\begin{align*}
&\int_{0}^T \int_{\mathbb{R}^{D}}\int_{\mathbb{R}^{D}} \Big(|v-\bar{v}|^2  + \lambda |f|^2\Big)d\mu(t,x,v)~dxdt\\[1mm]
 \leq &~ \liminf_{k\rightarrow \infty}\int_0^T\int_{\mathbb{R}^{D}}\int_{\mathbb{R}^{D}}  \Big(|v-\bar{v}|^2  + \lambda |f_{N_k}|^2\Big)d\mu_{N_k}(t,x,v) ~dxdt.    
\end{align*}
\end{theorem}

\subsection{Optimal control problem for the hydrodynamic equation}\label{section2.3}

Now we formally derive the corresponding hydrodynamic equation using moments \cite{Lev96a}. Multiplying the mean field equation in \eqref{mean-field-pb} by $1,v,|v|^2/2$ and integrating over $v\in\mathbb{R}^D$ yield
\begin{align*}
\partial_t \rho + \nabla_x \cdot (\rho u) &= 0,\\[2mm]
\partial_t (\rho u) + 
     \nabla_x \cdot ( \rho u \otimes u + P) &= S_1,\\[2mm]
 \partial_t (\rho E) + 
     \nabla_x \cdot ( \rho E u + P u + q) &= S_2,
\end{align*}
where the macroscopic quantities are defined by
\begin{align*}
\rho &=\int_{\mathbb{R}^D} d\mu_{t,x}(v),\qquad \rho u_i=\int_{\mathbb{R}^D}v_i ~d\mu_{t,x}(v) ,\qquad \rho E=\int_{\mathbb{R}^D}\frac{1}{2}|v|^2 d\mu_{t,x}(v),\\[1mm]
P_{ij}&=\int_{\mathbb{R}^D}(v_i-u_i)(v_j-u_j) d\mu_{t,x}(v),\qquad
q_{i}=\int_{\mathbb{R}^D}(v_i-u_i)|v-u|^2 d\mu_{t,x}(v),
\end{align*}
respectively. Here we write $\mu_{t,x}(v)=\mu(t,x,v)$ to represent $\mu$ as a distribution of $v$. The source terms are given by
\begin{align*}
    S_1=Q_1+\int_{\mathbb{R}^{D}} f\mu ~dv,\qquad S_2=Q_2+\int_{\mathbb{R}^{D}} v\cdot f\mu ~dv
\end{align*}
with
\begin{align}
    Q_1&=\int_{\mathbb{R}^{D}} \Psi(x,y)\rho(t,x)\rho(t,y)\Big[u(t,y)-u(t,x)\Big]dy,\\[2mm]
    Q_2&= \int_{\mathbb{R}^{D}} \Psi(x,y)\rho(t,x)\rho(t,y)\Big[
u(t,x)\cdot u(t,y)-E(t,x)-E(t,y)\Big] dy.
\end{align}
We assume all quantities are finite. Note that the above moment system is not closed since the equations for $\rho u$ and $\rho E$ also depends on higher order moments $P$ and $q$. In this work, we discuss the optimal control problems with two basic ansatz for closure. Namely, we take two different ansatz of $\mu$ to compute the higher order moments $P$ and $q$ and close the system. 

\subsubsection{Type I:}\label{section2.3.1}
Firstly, we use the equilibrium closure 
$$
\mu_1(t,x,v) = \rho(t,x)\delta(v-u(t,x)).
$$
Then it is not difficult to derive the pressure-less equation with control: 
\begin{align*}
    (\rho u)_t + \nabla_x \cdot (\rho u \otimes u) &= Q_1+\rho(t,x) f_H(t,x).
\end{align*}
Here we use the relation $\int f(t,x,v)d \mu_{t,x}(v) = \rho(t,x)f(t,x,u(t,x))$ to get $f_H(t,x)=f(t,x,u(t,x))$.
Moreover, we substitute the ansatz into the cost function in \eqref{mean-field-pb} and formally obtain the corresponding problem $\mathcal{Q}_{H_1}(0,T,\rho_0,u_0)$:
\begin{align}
\mathcal{V}_{H_1}(0,T,\rho_0,u_0)&=\min_{f_H}   \int_{0}^T \int_{\mathbb{R}^{D}} \rho|u-\bar{v}|^2 + \lambda \rho|f_H|^2 ~dxdt. \label{hydro1-pb1}\\[2mm]
&\rho_t + \nabla_x \cdot (\rho u) = 0,\label{hydro1-pb2}\\[2mm]
    &(\rho u)_t + \nabla_x \cdot (\rho u \otimes u) = Q_1+\rho f_H,\label{hydro1-pb3}\\[2mm] 
    &\rho(0,x)=\rho_0(x),\quad u(0,x)=u_0(x).\label{hydro1-pb4}
\end{align}
In this work, we will make the strong assumption that the optimal control problem \eqref{hydro1-pb1}---\eqref{hydro1-pb4} has the optimal solution $(\rho,u)$ with the optimal control $f_H$. Specifically, for the initial data $(\rho_0,u_0)$ given in a compact support $\mathcal{U}_0$, we assume that there is a compact support $\mathcal{U}_T$, depending on $T$, such that the solution $(\rho,u)(t)\in \mathcal{U}_T$ for all $t\in [0,T]$. Moreover, the solution satisfies $(\rho,u)\in C([0,T];H^1(\mathbb{R}^D))$ and 
$\rho(t,x)>0$ point-wise.

\subsubsection{Type II:}\label{section2.3.2}
We consider the Maxwellian
$$
d\mu_{2,t,x}(v) = \frac{\rho}{(2\pi \theta)^{D/2}}\exp\left(-\frac{|v-u|^2}{2\theta}\right)dv.
$$
This gives the Euler equations with
$$
P_{ij} = \int_{\mathbb{R}^D}(v_i-u_i)(v_j-u_j)\mu_2dv = \delta_{ij} \rho \theta,\qquad q_{i}=\int_{\mathbb{R}^D}(v_i-u_i)|v-u|^2\mu_2 dv = 0.
$$
We denote the pressure term and the energy by
$$
p=\rho \theta,\qquad 
\rho E = \rho e + \frac{1}{2}\rho|u|^2 = \frac{D}{2}\rho \theta  + \frac{1}{2}\rho|u|^2.
$$
Furthermore, the source terms are
$$
S_1=Q_1+\rho F_1,\qquad S_2=Q_2+\rho F_2
$$
with $F_1$ and $F_2$ defined by
\begin{align*}
F_1(t,x)&=\frac{1}{(2\pi \theta)^{D/2}} \int_{\mathbb{R}^D} f(t,x,v)\exp\left(-\frac{|v-u|^2}{2\theta}\right) dv\\[2mm]
F_2(t,x)&=\frac{1}{(2\pi \theta)^{D/2}} \int_{\mathbb{R}^D} v\cdot f(t,x,v)\exp\left(-\frac{|v-u|^2}{2\theta}\right) dv.
\end{align*}
Moreover, the cost function reads 
\begin{align*}
&~ \int_{0}^T\int_{\mathbb{R}^{D}}\int_{\mathbb{R}^{D}} |v-\bar{v}|^2 + \lambda |f|^2 ~d\mu_{2,t,x}(v) dx dt \\[2mm]
=&~ \int_{0}^T\int_{\mathbb{R}^{D}}\int_{\mathbb{R}^{D}} |v-u+u-\bar{v}|^2 + \lambda |f|^2 ~d\mu_{2,t,x}(v) dx dt \\[2mm]
=&~ \int_{0}^T\int_{\mathbb{R}^{D}}\int_{\mathbb{R}^{D}} |v-u|^2+|u-\bar{v}|^2 + \lambda |f|^2 ~d\mu_{2,t,x}(v) dx dt \\[2mm]
=&~ \int_{0}^T \int_{\mathbb{R}^{D}} \Big(\rho|u-\bar{v}|^2 + 2 \rho e\Big) dxdt+ \lambda \int_{0}^T\int_{\mathbb{R}^{D}}\int_{\mathbb{R}^{D}} |f|^2 ~d\mu_{2,t,x}(v) dx dt.
\end{align*}
Unlike the pressure-less case, the term $\int_{\mathbb{R}^{D}} |f|^2 ~d\mu_{2,t,x}(v)$ can not be expressed in terms of macroscopic control variables $F_1$ and $F_2$. Notice that, by using H\"older's inequality, it follows: 
\begin{align*}
\rho|F_1|^2&\leq \frac{\rho}{(2\pi \theta)^{D}} \left(\int_{\mathbb{R}^D} |f|^2\exp\left(-\frac{|v-u|^2}{2\theta}\right) dv\right) \left(\int_{\mathbb{R}^D} \exp\left(-\frac{|v-u|^2}{2\theta}\right) dv\right)\\[2mm]
&\leq \int_{\mathbb{R}^{D}} |f|^2 ~d\mu_{2,t,x}(v)
\end{align*}
and
\begin{align*}
\rho|F_2-u\cdot F_1|^2&=\frac{\rho}{(2\pi \theta)^{D}} \left(\int_{\mathbb{R}^D} (v-u)\cdot f(t,x,v)\exp\left(-\frac{|v-u|^2}{2\theta}\right) dv\right)^2\\[2mm]
&\leq \frac{\rho}{(2\pi \theta)^{D}} \left(\int_{\mathbb{R}^D} |f|^2\exp\left(-\frac{|v-u|^2}{2\theta}\right) dv\right) \int_{\mathbb{R}^D} |v-u|^2\exp\left(-\frac{|v-u|^2}{2\theta}\right) dv\\[2mm]
&= D\theta  \int_{\mathbb{R}^{D}} |f|^2 ~d\mu_{2,t,x}(v)= 2e  \int_{\mathbb{R}^{D}} |f|^2 ~d\mu_{2,t,x}(v).
\end{align*}
Hence, minimizing the cost $\int_{\mathbb{R}^{D}} |f|^2\mu_2(t,x,v)dv$ implies the minimization of both $\rho|F_1|^2$ and $\dfrac{\rho}{2e}|F_2-u\cdot F_1|^2$.
Motivated by this, we define a modified objective function with
\begin{align}\label{def:G}
    G(F_1,F_2) = \rho|F_1|^2 + \frac{\rho}{2e}|F_2-u\cdot F_1|^2
\end{align}
and consider the optimal control problem
\begin{align}
\mathcal{V}_{H_2}(0,T,\rho_0,u_0,E_0)
&=\min_{F_1,F_2}   \int_{0}^T \int_{\mathbb{R}^{D}} \rho|u-\bar{v}|^2 + 2 \rho e + \lambda G(F_1,F_2) ~dxdt.\label{hydro2-pb1}\\[2mm]
\partial_t \rho &+ \nabla_x \cdot (\rho u) = 0,\label{eq3:rho}\\[2mm]
\partial_t (\rho u) &+ 
     \nabla_x \cdot ( \rho u \otimes u + p I) = Q_1+\rho F_1,\label{eq3:m}\\[2mm]
\partial_t (\rho E) &+ 
     \nabla_x \cdot (\rho E u + p u) = Q_2 + \rho F_2,\label{eq3:E}\\[2mm] 
    \rho(0,x)&=\rho_0(x),\quad u(0,x)=u_0(x),\quad E(0,x)=E_0(x).\label{hydro2-pb5}
\end{align}
\begin{remark}
In this work, we take $F_1$ and $F_2$ as two independent controls for the hydrodynamic equations. Note that the  hydrodynamic optimal control problem derived from the kinetic problem is not unique. In Section \ref{section3.4} and \ref{Section4}, we show that, if it exists, the solution of our optimal control problem satisfies the turnpike property.
\end{remark}

Similarly, we assume that the optimal control problem \eqref{hydro2-pb1}---\eqref{hydro2-pb5} has the optimal solution $(\rho,u,E)$ with the optimal control $F_1$ and $F_2$. Given the initial data $(\rho_0,u_0,E_0)$ in a compact support $\mathcal{U}_0$, we assume that there is a compact support $\mathcal{U}_T$, depending on $T$, such that the solution $(\rho,u,E)(t)\in \mathcal{U}_T$ for all $t\in [0,T]$. Moreover, the solution satisfies $(\rho,u,E)\in C([0,T];H^1(\mathbb{R}^D))$ and 
$\rho(t,x)>0, e(t,x)>0$ point-wise. 

\section{Feedback control and cheap control property}\label{Section3}

In this section, we investigate the cheap control property, indicating that the optimal trajectories remain close to steady states on average over sufficiently long time horizons. We first construct suitable feedback control laws for the particle system and then extend these constructions to the mean-field and hydrodynamic levels. Based on these feedback laws, we establish the cheap control property, which can be interpreted as an integral turnpike property. Moreover, the cheap control property serves as an important step toward establishing the exponential turnpike property in the next section. Subsections \ref{Section3.1}–\ref{Section3.3} are devoted to the details for the particle system, the mean-field system, and the hydrodynamic models, respectively.

\subsection{Particle systems}\label{Section3.1}
We start with the following cheap control property for the optimal control problem \eqref{particle-eq} and \eqref{particle:cost}.
\begin{lemma}[Cheap control] \label{lemma1.1}
For the optimal solution $(x(t),v(t))$, it holds
$$
\int_a^T \|v(t)-\bar{v}\|_N^2 + \lambda\|f\|^2_N ~dt \leq \sqrt{\lambda} 
\|v(a)-\bar{v}\|_N^2
$$
for any $0\leq a\leq T$.
\end{lemma}
\begin{proof}
    Considering the feedback control 
    $$
\widetilde{f}_i(t) = - \beta (\widetilde{v}_i-\bar{v})
    $$
with $\beta$ a positive constant and $\widetilde{v}_i$ the associated solution, we have 
$$
\frac{d\widetilde{v}_i}{dt} =  \frac{1}{N}\sum_{j=1}^N \Psi(\widetilde{x}_i,\widetilde{x}_j)(\widetilde{v}_j-\widetilde{v}_i) - \beta (\widetilde{v}_i-\bar{v}).
$$
Multiplying $\widetilde{v}_i-\bar{v}$ and summing up $i$ yield
\begin{align*}
    \frac{d\|\widetilde{v}-\bar{v}\|_N^2}{dt} =&~  \frac{2}{N^2}\sum_{i,j=1}^N \Psi(\widetilde{x}_i,\widetilde{x}_j)\Big\langle \widetilde{v}_i-\bar{v}, (\widetilde{v}_j-\bar{v}) - (\widetilde{v}_i-\bar{v}) \Big\rangle - 2\beta \|\widetilde{v}-\bar{v}\|_N^2.
\end{align*}
Since the function $\Psi$ is symmetric with respect to the indices $i$ and $j$, we know that
\begin{align*}
    \sum_{i,j=1}^N\Psi(\widetilde{x}_i,\widetilde{x}_j)\langle \widetilde{v}_i-\bar{v}, \widetilde{v}_j-\widetilde{v}_i \rangle =&- \sum_{i,j=1}^N\Psi(\widetilde{x}_i,\widetilde{x}_j)\langle \widetilde{v}_j-\bar{v}, \widetilde{v}_j-\widetilde{v}_i \rangle \\[2mm]
    =&-\frac{1}{2} \sum_{i,j=1}^N\Psi(\widetilde{x}_i,\widetilde{x}_j)\langle \widetilde{v}_j-\widetilde{v}_i, \widetilde{v}_j-\widetilde{v}_i \rangle \leq 0.
\end{align*}
Therefore, we obtain 
\begin{align*}
    \frac{d\|\widetilde{v}-\bar{v}\|_N^2}{dt} 
    \leq &~ - 2\beta \|\widetilde{v}-\bar{v}\|_N^2,
\end{align*}
which implies 
$$
\|\widetilde{v}-\bar{v}\|_N^2 \leq e^{-2\beta (t-a)} \|v(a)-\bar{v}\|_N^2.
$$
Then we compute 
\begin{align*}
\mathcal{V}_N(a,T,x(a),v(a)) \leq &~
\int_{a}^T \|\widetilde{v}-\bar{v}\|_N^2 + \lambda \|\widetilde{f}(t)\|_N^2 dt  = \int_{a}^T (1+ \lambda \beta^2)\|\widetilde{v}-\bar{v}\|_N^2 dt\\
\leq &~
\frac{1+\lambda \beta^2}{2\beta} \|v(a)-\bar{v}\|_N^2 .
\end{align*}
Taking $\beta = 1/\sqrt{\lambda}$, we have the bound $(1+\lambda \beta^2)/(2\beta) = \sqrt{\lambda}$.

\end{proof}

\begin{remark}
    The proof of the cheap control property is based on the design of proper feedback controls. The idea of the proof is similar to the first-order particle systems in \cite{Herty_Zhou_2025}.
\end{remark}

\subsection{Mean field level}\label{Section3.2}

Recall that the particle system \eqref{particle-eq} is related to the mean field equation \eqref{mean-field-pb} by considering the empirical measure 
\begin{equation*}
\mu_{N}(t,x,v) = \frac{1}{N}\sum_{i=1}^N \delta\left(x-x_i(t)\right)\otimes \delta\left(v-v_i(t)\right).
\end{equation*}
The feedback law and the cheap control property in the previous subsection are valid for each fix number $N$. Now, we discuss the limiting behavior as $N\rightarrow \infty$.

\begin{lemma}[Cheap control]
Suppose $(\mu(t,x,v),f(t,x,v))$ is the  solution to the optimal control problem \eqref{mean-field-pb}, then the following inequality holds:
\begin{align}
\int_{0}^T \int_{\mathbb{R}^{D}} \Big(|v-\bar{v}|^2  + \lambda |f|^2\Big)d\mu(t,x,v)~dt
\leq &~ \sqrt{\lambda} \int_{\mathbb{R}^{D}} |v-\bar{v}|^2 d\mu(0,x,v).\label{cheap-w} 
\end{align}    
\end{lemma}

Note that the result in this lemma is analogous to the estimate in the discrete case if we formally take the limit  $N\rightarrow \infty$ in Lemma \ref{lemma1.1}. 

\begin{proof}
Due to lower semi-continuity, we have 
\begin{align*}
\int_{0}^T \int_{\mathbb{R}^{D}} \Big(|v-\bar{v}|^2  + \lambda |f|^2\Big)d\mu(t,x,v)~dt
 \leq&~ \liminf_{k\rightarrow \infty}\int_0^T \Big(|v-\bar{v}|^2  + \lambda |f_{N_k}|^2\Big)d\mu_{N_k}(t,x,v) dt \\[2mm] 
= &~ \liminf_{k\rightarrow \infty}\frac{1}{N_k}\sum_{i=1}^{N_k}\int_0^T |v_i-\bar{v}|^2+\lambda |f_{N_k}|^2 dt. 
\end{align*}
On the other hand, since $f_{N_k}$ is the optimal solution to \eqref{particle-eq} and \eqref{particle:cost}, it follows from the cheap control property of the particle system that
\begin{align*}
\liminf_{k\rightarrow \infty}\frac{1}{N_k}\sum_{i=1}^{N_k}\int_0^T |v_i-\bar{v}|^2+\lambda |f_{N_k}|^2 dt
\leq&~ \liminf_{k\rightarrow \infty} \sqrt{\lambda} \frac{1}{N_k}\sum_{i=1}^{N_k} |v_i(0)-\bar{x}|^2 \\[2mm]
=&~ \sqrt{\lambda} \int |v-\bar{v}|^2 d\mu(0,x,v).
\end{align*}
Here, $\sqrt{\lambda}$ is the same constant as that in Lemma \ref{lemma1.1}.
\end{proof} 

\subsection{Hydrodynamic level}\label{Section3.3}

We now turn to the analysis of the cheap control property at the hydrodynamic level. Recall that two different closure ansatz are employed in this setting. Accordingly, we study two optimal control problems in the following subsections.

\subsubsection{Pressure-less Euler equation}
In this subsection, we consider feedback law and the cheap control property for the optimal control problem \eqref{hydro1-pb1}---\eqref{hydro1-pb4}.
Define the energy functional
$$
\mathcal{E}(t) = \frac{1}{2}\int_{\mathbb{R}^D}\rho(t,x)|u(t,x)-\bar{v}|^2dx.
$$
We first state 
\begin{lemma}\label{lemma31}
    For any control function $f=f(t,x)$ with strictly positive density $\rho=\rho(t,x)$, the energy functional satisfies the equality 
    $$
    \mathcal{E}(t_2) \leq \mathcal{E}(t_1) +  \int_{t_1}^{t_2}\int_{\mathbb{R}^D} \rho (u-\bar{v})\cdot f_H~ dxdt
    $$
    with $0\leq t_1 < t_2 \leq T$.
\end{lemma}
\begin{proof}
A direct computation shows 
$$
\frac{d\mathcal{E}}{dt} = \frac{1}{2}\frac{d}{dt}\int_{\mathbb{R}^D}\rho(t,x)|u(t,x)|^2dx - \frac{d}{dt} \int_{\mathbb{R}^D}\rho(t,x)u(t,x)\cdot\bar{v} dx + \frac{1}{2}\frac{d}{dt}\int_{\mathbb{R}^D}\rho(t,x)|\bar{v}|^2dx.
$$
Due to the conservative of mass, the last term is zero and we have
\begin{align}\label{fluid-time-derivative-E}
\frac{d\mathcal{E}}{dt} = \frac{1}{2}\frac{d}{dt}\int_{\mathbb{R}^D}\rho(t,x)|u(t,x)|^2dx - \frac{d}{dt}\int_{\mathbb{R}^D}\rho(t,x)u(t,x)\cdot\bar{v} dx.
\end{align}
We use the weak form of the equation \eqref{hydro1-pb2}---\eqref{hydro1-pb3}. For any test function $(\phi,\psi)$, it follows that
\begin{align}
    \int_{\mathbb{R}^D} (\phi {\rho} + \psi \cdot {m}) ~dx \bigg|_{t=t_1}^{t=t_2} = &~\int_{t_1}^{t_2}\int_{\mathbb{R}^D}  (\phi_t{\rho}+\psi_t \cdot {m})~dxdt\nonumber \\[2mm]
    +&\int_{t_1}^{t_2}\int_{\mathbb{R}^D} \nabla_x \phi \cdot {m}
    + \nabla_x \psi : \left(\frac{{m}\otimes {m}}{{\rho}}\right) + \psi \cdot S_1~ dxdt\label{fluid-weak-eq}
\end{align}
with $S_1=Q_1+\rho f_H$.
In particular, we take 
$\phi = -\dfrac{|u|^2}{2},\psi = u$
and compute 
\begin{align*}
    \int_{\mathbb{R}^D} (\phi {\rho} + \psi \cdot {m}) ~dx \bigg|_{t=t_1}^{t=t_2} = \left( \int_{\mathbb{R}^D} \frac{1}{2}{\rho}|u|^2 ~dx\right) \bigg|_{t=t_1}^{t=t_2}.
\end{align*}
On the other hand, we compute 
\begin{align*}
\int_{t_1}^{t_2}\int_{\mathbb{R}^D}  (\phi_t{\rho}+\psi_t{m})~dxdt = \int_{t_1}^{t_2}\int_{\mathbb{R}^D}  (-{\rho}u\cdot u_t+u_t\cdot {m})~dxdt=0
\end{align*}
and
\begin{align*}
    &~\int_{t_1}^{t_2}\int_{\mathbb{R}^D} \nabla_x \phi \cdot {m}
    + \nabla_x \psi : \left(\frac{{m}\otimes {m}}{{\rho}}\right) + \psi \cdot S_1~ dxdt\\[2mm] 
    = &~  \int_{t_1}^{t_2}\int_{\mathbb{R}^D}  -\nabla_x u \cdot u \cdot {m} + \nabla_x u : \left(\rho{u}\otimes {u}\right)
    + u\cdot S_1 ~dxdt 
    = \int_{t_1}^{t_2}\int_{\mathbb{R}^D} u\cdot S_1 ~dxdt.
\end{align*}
The last three equations give
\begin{align*}
    \left( \int_{\mathbb{R}^D} \frac{1}{2}{\rho}|u|^2 ~dx\right) \bigg|_{t=t_1}^{t=t_2} = \int_{t_1}^{t_2}\int_{\mathbb{R}^D} u\cdot S_1 ~dxdt.
\end{align*}
Moreover, we take $\phi(x) = -\dfrac{|\bar{v}|^2}{2}\chi_{B(R)}(x)$ and $\psi(x) = \bar{v}~\chi_{B(R)}(x)$ in \eqref{fluid-weak-eq}. Here $\chi_{B(R)}(x)$ means the characteristic function of the ball $B(R)=\{|x|<R\}$. According to our assumption that the solution $(\rho,u)$ belongs to a compact support, we may take $R$ sufficiently large and obtain
\begin{align*}
    \left( \int_{\mathbb{R}^D} {\rho}u\cdot \bar{v} ~dx\right) \bigg|_{t=t_1}^{t=t_2} 
    = &~ \int_{t_1}^{t_2} \int_{\mathbb{R}^D} \bar{v}\cdot S_1 ~dxdt.
\end{align*}
Combining these two equations with \eqref{fluid-time-derivative-E}, we have
\begin{align*}
    &\mathcal{E}(t_2)-\mathcal{E}(t_1) = \int_{t_1}^{t_2}\int_{\mathbb{R}^D} (u-\bar{v})\cdot S_1~ dxdt
    =\int_{t_1}^{t_2}\int_{\mathbb{R}^D}Q_1dx dt
    + \int_{t_1}^{t_2} \int_{\mathbb{R}^D} \rho (u-\bar{v})\cdot f_H dxdt.
\end{align*}
For the first term on the right hand side, we use the symmetric property of $\Psi(x,y)$ to compute
\begin{align*}
\int_{t_1}^{t_2}\int_{\mathbb{R}^D}Q_1dx dt=&\int_{t_1}^{t_2}\int_{\mathbb{R}^D}\int_{\mathbb{R}^D}\Psi(x,y)\rho(x)\rho(y) (u(x)-\bar{v})\cdot(u(y)-u(x)) dxdydt\\[2mm]    
=&-\int_{t_1}^{t_2}\int_{\mathbb{R}^D}\int_{\mathbb{R}^D}\Psi(x,y)\rho(x)\rho(y) (u(y)-\bar{v})\cdot(u(y)-u(x)) dxdydt\\[2mm]
=&-\frac{1}{2}\int_{t_1}^{t_2}\int_{\mathbb{R}^D}\int_{\mathbb{R}^D}\Psi(x,y)\rho(x)\rho(y) |u(y)-u(x)|^2 dxdydt \leq 0.
\end{align*}
Then we obtain the inequality stated in the lemma and complete the proof.
\end{proof}
By using this lemma, we can use the feedback control to conclude
\begin{lemma}\label{lemma3.4}[Cheap control]
For the optimal control $f_H$ with the associated solution $(\rho,u)$, the following cheap control property holds for any $0\leq a \leq T$:
$$
\int_a^{T}\int_{\mathbb{R}^D} \rho |u-\bar{v}|^2 + \lambda \rho |f_H|^2 dxdt\leq \sqrt{\lambda} \int_{\mathbb{R}^D} \rho(a,x) |u(a,x)-\bar{v}|^2 dx.
$$
\end{lemma}
\begin{proof}
Starting from $t=a$, we take the feedback control $\widetilde{f}_H=-\beta(\widetilde{u}-\bar{v})$ where $\widetilde{\rho}$ and $\widetilde{u}$ are solutions to the equation 
\begin{align*}
&\widetilde{\rho}_t + \nabla_x \cdot (\widetilde{\rho} \widetilde{u}) = 0,\\[2mm]
&(\widetilde{\rho} \widetilde{u})_t + \nabla_x \cdot (\widetilde{\rho} \widetilde{u} \otimes \widetilde{u}) = S_1,\\[2mm] 
&\widetilde{\rho}(a,x)=\rho(a,x),\quad \widetilde{u}(a,x)=u(a,x).
\end{align*}
Note that $\rho(a,x)$ and $u(a,x)$ are values of optimal trajectory at time $t=a$.
Due to Lemma \ref{lemma31}, we have
    $$
    \mathcal{E}(t) - \mathcal{E}(a) \leq -\beta  \int_{a}^{t}\int_{\mathbb{R}^D} \widetilde{\rho} |\widetilde{u}-\bar{v}|^2 dxdt=-2\beta\int_a^t \mathcal{E}(s)ds
    $$
    for any $a\leq t\leq T$. By using the Gronwall's inequality, we have
$$
\mathcal{E}(t) \leq  \exp(-2\beta(t-a))\mathcal{E}(a)
$$
which can be also written as
$$
\int_{\mathbb{R}^D} \widetilde{\rho}(t,x) |\widetilde{u}(t,x)-\bar{v}|^2 dx \leq  \exp(-2\beta(t-a))\int_{\mathbb{R}^D} \rho(a,x) |u(a,x)-\bar{v}|^2 dx.
$$
Now we consider the optimal control problem $\mathcal{V}_H(a,T,\rho(x,a),u(x,a))$ which starts from time $t=a$. Clearly, $f_H|_{t\in[a,T]}$ is also the optimal control for this subproblem. 
For the optimal control, we know that 
\begin{align*}
\int_a^{T}\int_{\mathbb{R}^D} \rho |u-\bar{v}|^2 + \lambda \rho |f_H|^2 dxdt&\leq  \int_a^{T}\int_{\mathbb{R}^D} \widetilde{\rho} |\widetilde{u}-\bar{v}|^2 + \lambda \widetilde{\rho} |\widetilde{f}|^2 dxdt\\[1mm]
&= (1+\lambda\beta^2)\int_a^{T}\int_{\mathbb{R}^D} \widetilde{\rho} |\widetilde{u}-\bar{v}|^2 dxdt\\[1mm]
&\leq (1+\lambda\beta^2)\int_a^{T}e^{-2\beta(t-a)}dt \int_{\mathbb{R}^D} \rho(a,x) |u(a,x)-\bar{v}|^2 dx\\[1mm]
&\leq \frac{1+\lambda\beta^2}{2\beta} \int_{\mathbb{R}^D} \rho(a,x) |u(a,x)-\bar{v}|^2 dx.
\end{align*}
At last, we take $\beta=1/\sqrt{\lambda}$ and complete the proof.
\end{proof}

\subsubsection{The full Euler equation}\label{section3.4}

Similarly, we can also consider feedback law and the cheap control property for the optimal control problem \eqref{hydro2-pb1}---\eqref{hydro2-pb5}. In analogy with Lemma \ref{lemma31}, we prove the following lemma

\begin{lemma}\label{lemma3.5} For any control functions $F_1=F_1(t,x)$ and $F_2=F_2(t,x)$ such that the corresponding solution has positive density $\rho=\rho(t,x)$ and positive energy $e=e(t,x)$, we have the following inequality
\begin{align*}
    \int_{\mathbb{R}^D} \frac{1}{2}\rho|u-\bar{v}|^2 +  \rho e ~dx\Big|_{t=t_1}^{t=t_2} \leq \int_{t_1}^{t_2}\int_{\mathbb{R}^D}\rho (F_2 - \bar{v}\cdot F_1)~dxdt.
\end{align*}
\end{lemma}
\begin{proof}

Firstly, we integrate $t\in[t_1,t_2]$ in the equation  \eqref{eq3:E} to get 
\begin{align*}
    \int_{\mathbb{R}^D} \frac{1}{2}\rho|u|^2 +  \rho e ~dx\Big|_{t=t_1}^{t=t_2} = \int_{t_1}^{t_2}\int_{\mathbb{R}^D}S_2~dxdt=\int_{t_1}^{t_2}\int_{\mathbb{R}^D}Q_2+\rho F_2~dxdt.
\end{align*}
Notice that the source term satisfies 
\begin{align*}
\int_{\mathbb{R}^D}Q_2~dx =& - \int_{\mathbb{R}^{2D}} \Psi(x,y)\rho(t,x)\rho(t,y)[E(t,x)+E(t,y)-u(t,x)\cdot u(t,y)] dxdy \\[1mm] 
=& - \int_{\mathbb{R}^{2D}} \Psi(x,y)\rho(t,x)\rho(t,y)\left[e(t,x)+e(t,y)+\frac{1}{2}|u(t,x)-u(t,y)|^2\right] dxdy \\[1mm]
\leq&~ 0.
\end{align*}
Then we have 
\begin{align}\label{eq3:test2}
    \int_{\mathbb{R}^D} \frac{1}{2}\rho|u|^2 +  \rho e~dx\Big|_{t=t_1}^{t=t_2} \leq  \int_{t_1}^{t_2}\int_{\mathbb{R}^D}\rho F_2~dxdt.
\end{align}
On the other hand, we consider the weak form of equations \eqref{eq3:rho}, \eqref{eq3:m}: 
\begin{align}
    &\int_{\mathbb{R}^D} \rho \phi ~dx\Big|_{t=t_1}^{t=t_2} = \int_{t_1}^{t_2}\int_{\mathbb{R}^D}\Big(\rho \phi_t + (\rho u)\cdot \nabla \phi \Big)dxdt\label{eq3:rho-weak}\\[2mm]
    &\int_{\mathbb{R}^D} \rho u\cdot \psi~ dx\Big|_{t=t_1}^{t=t_2} = \int_{t_1}^{t_2}\int_{\mathbb{R}^D}\rho u\cdot \psi_t + (\rho u\otimes u):\nabla \psi  + p\nabla\cdot \psi +\psi \cdot S_1~dxdt\label{eq3:m-weak}
\end{align}
and take the test functions $\phi=\frac{1}{2}|\bar{v}|^2 \chi_{B(R)}(x),~\psi=-\bar{v}~\chi_{B(R)}(x)$
to obtain
\begin{align*}
    \int_{\mathbb{R}^D} \frac{1}{2}\rho|\bar{v}|^2 - \rho u\cdot \bar{v}~dx\Big|_{t=t_1}^{t=t_2} = \int_{t_1}^{t_2}\int_{\mathbb{R}^D}-\bar{v}\cdot S_1~dxdt= -\bar{v}\cdot \int_{t_1}^{t_2}\int_{\mathbb{R}^D}Q_1+\rho F_1~dxdt.
\end{align*}
Here we also use the assumption that the solution $(\rho,u,E)$ belongs to a compact support $\mathcal{U}_T\subset B(R)$ for sufficiently large $R$.
Due to the symmetric property of $\Psi(x,y)$, we have
\begin{align*}
    \int_{\mathbb{R}^D}Q_1~dx &= - \int_{\mathbb{R}^{2D}} \Psi(x,y)\rho(t,x)\rho(t,y)[u(t,x)-u(t,y)]dxdy=0.
\end{align*}
Thus we obtain
\begin{align*}
    \int_{\mathbb{R}^D} \frac{1}{2}\rho|\bar{v}|^2 - \rho u\cdot \bar{v}~dx\Big|_{t=t_1}^{t=t_2} = - \int_{t_1}^{t_2}\int_{\mathbb{R}^D}\rho \bar{v}\cdot F_1~dxdt.
\end{align*}
Combining this with the inequality in \eqref{eq3:test2}, we get
\begin{align}\label{eq3:test3}
    \int_{\mathbb{R}^D} \frac{1}{2}\rho|u-\bar{v}|^2 +  \rho e ~dx\Big|_{t=t_1}^{t=t_2} \leq \int_{t_1}^{t_2}\int_{\mathbb{R}^D}\rho (F_2 - \bar{v}\cdot F_1)~dxdt.
\end{align}
This completes the proof of lemma. 
\end{proof}

Then we use this lemma to obtain

\begin{lemma}[Cheap control]\label{lemma3.6}
For the optimal control $F_1$ and $F_2$ with the associated solution $(\rho,u,E)$, the following cheap control property holds for any $0\leq a \leq T$:
$$
\int_a^{T}\int_{\mathbb{R}^D} \rho |u-\bar{v}|^2 + 2\rho e + \lambda G(F_1,F_2) dxdt\leq \sqrt{\lambda} \int_{\mathbb{R}^D} \rho(a,x) |u(a,x)-\bar{v}|^2 + 2\rho(a,x)e(a,x)dx.
$$
\end{lemma}
\begin{proof}
We also consider the feedback control 
$$
\widetilde{F}_1=-\beta(\widetilde{u}-\bar{v}),\qquad \widetilde{F}_2=-\beta[2\widetilde{e}+\widetilde{u}\cdot(\widetilde{u}-\bar{v})]
$$ in the equation 
\begin{align*}
\partial_t \widetilde{\rho} + \nabla_x \cdot (\widetilde{\rho} \widetilde{u}) &= 0,\\[2mm]
\partial_t (\widetilde{\rho} \widetilde{u}) + 
     \nabla_x \cdot ( \widetilde{\rho} \widetilde{u} \otimes \widetilde{u} + \widetilde{p} I) &= S_1,\\[2mm]
\partial_t (\widetilde{\rho} \widetilde{E}) + 
     \nabla_x \cdot (\widetilde{\rho} \widetilde{E} \widetilde{u} + \widetilde{p} \widetilde{u}) &= S_2,\\[1mm] 
    \widetilde{\rho}(a,x)=\rho(a,x),\quad \widetilde{u}(a,x)=&~u(a,x),\quad \widetilde{E}(a,x)=E(a,x).
\end{align*}
Here $\rho(a,x), u(a,x)$ and $E(a,x)$ are values of optimal trajectory at time $t=a$.
We use Lemma \ref{lemma3.5} to obtain
\begin{align*}
    \int_{\mathbb{R}^D} \frac{1}{2}\widetilde{\rho}|\widetilde{u}-\bar{v}|^2 +  \widetilde{\rho} \widetilde{e} ~dx\Big|_{t=a}^{t=s} 
    \leq &~\int_{a}^{s}\int_{\mathbb{R}^D}\widetilde\rho (\widetilde{F}_2 - \bar{v}\cdot \widetilde{F}_1)~dxdt\\[2mm]
    = &~-2\beta \int_{a}^{s}\int_{\mathbb{R}^D}\widetilde \rho \widetilde{e} + \frac{1}{2}\widetilde{\rho}|\widetilde{u} - \bar{v}|^2~dxdt
\end{align*}
    for any $a\leq s\leq T$. By using the Gronwall's inequality, we have
$$
\mathcal{H}(t) \leq  \exp(-2\beta(t-a))\mathcal{H}(a)
$$
with 
$$
\mathcal{H}(t)=\int_{\mathbb{R}^D} \widetilde{\rho}(t,x)\widetilde{e}(t,x)+\frac{1}{2}\widetilde{\rho}(t,x) |\widetilde{u}(t,x)-\bar{v}|^2 dx.
$$
Now we consider the optimal control problem $\mathcal{V}_{H_2}(a,T,\rho(x,a),u(x,a),E(x,a))$ which starts from time $t=a$. Clearly, $(F_1,F_2)|_{t\in[a,T]}$ is also the optimal control for this subproblem. 
For the optimal control, we know that 
\begin{align*}
\int_a^{T}\int_{\mathbb{R}^D} \rho |u-\bar{v}|^2 + 2\rho e+\lambda G(F_1,F_2) dxdt&\leq  \int_a^{T}\int_{\mathbb{R}^D} \widetilde{\rho} |\widetilde{u}-\bar{v}|^2 + 2\widetilde{\rho} \widetilde{e}+\lambda G(\widetilde{F}_1,\widetilde{F}_2) dxdt.
\end{align*}
By the definition of $G$ in \eqref{def:G}, we compute
\begin{align*}
G(\widetilde{F}_1,\widetilde{F}_2) &= \rho|\widetilde{F}_1|^2 + \frac{\widetilde{\rho}}{2\widetilde{e}}|\widetilde{F}_2-\widetilde{u}\cdot \widetilde{F}_1|^2 
= \beta^2 \widetilde{\rho} |\widetilde{u}-\bar{v}|^2 + 2\beta^2 \widetilde{\rho}\widetilde{e}. 
\end{align*}
Then it follows that 
\begin{align*}
&\int_a^{T}\int_{\mathbb{R}^D} \rho |u-\bar{v}|^2 + 2\rho e+\lambda G(F_1,F_2) dxdt
\leq 2(1+\lambda\beta^2)\int_a^T\mathcal{H}(t)dt\\[1mm]
\leq&~ 2(1+\lambda\beta^2)\left(\int_a^{T}e^{-2\beta(t-a)}dt \right)\mathcal{H}(a)
\leq \frac{1+\lambda\beta^2}{\beta} \mathcal{H}(a)\\[1mm]
=&~ \frac{1+\lambda\beta^2}{2\beta} \int_{\mathbb{R}^D} \rho(a,x) |u(a,x)-\bar{v}|^2 + 2\rho(a,x)e(a,x)dx.
\end{align*}
At last, we take $\beta=1/\sqrt{\lambda}$ and complete the proof.
\end{proof}

\section{Exponential turnpike property}\label{Section4}

We are now in a position to state the main theorem.
Namely, we prove the exponential turnpike property: 

\begin{theorem}\label{thm14}
We have the following exponential turnpike results.
\begin{itemize}
\item[(1)] Particle system: there exist constants $C>0$ and $\alpha>0$, which are independent of $N$ and $T$, such that the optimal solution for the problem \eqref{particle-eq}---\eqref{particle:cost} satisfies the exponential turnpike property:
\begin{align}\label{expo-particle}
\|v(t)-\bar{v}\|_N^2  \leq C e^{-\alpha t} \|v(0)-\bar{v}\|_N^2
\end{align}
for any $t\in(0,T)$ when $T>0$ is sufficiently large. 
\item[(2)] Mean field problem: for the optimal control problem \eqref{mean-field-pb}, we assume that the initial data $\mu_0$ has bounded second-order moments with respect to $v$. Then we have the corresponding exponential turnpike property
\begin{equation}\label{turnpike-meanfield}
    \int_{\mathbb{R}^D\times \mathbb{R}^D}|v-\bar{v}|^2 d\mu(t,x,v)  \leq C e^{-\alpha t} \int_{\mathbb{R}^D\times \mathbb{R}^D}|v-\bar{v}|^2 d\mu_0(x,v)
\end{equation}
for any $t\in(0,T)$ when $T>0$ is sufficiently large. Here $C$ and $\alpha$ are the same constants as those in \eqref{expo-particle}.
\item[(3a)] Hydrodynamic equation (pressure-less case): suppose the problem \eqref{hydro1-pb1}---\eqref{hydro1-pb4} has the optimal solution $(\rho,u)$, which satisfies the property in Section \ref{section2.3.1}. Then the optimal solution satisfies the exponential turnpike property:
\begin{align}\label{eq4.3} 
\int_{\mathbb{R}^D} \rho(t,x) |u(t,x)-\bar{v}|^2 dx \leq C e^{-\alpha t} \int_{\mathbb{R}^D} \rho_0 |u_0-\bar{v}|^2 dx
\end{align}
for any $t\in(0,T)$ when $T>0$ is sufficiently large. \item[(3b)] Hydrodynamic equation (Euler equation): for the optimal control problem \eqref{hydro2-pb1}---\eqref{hydro2-pb5}, we assume the existence of  solution which satisfies the property in Section \ref{section2.3.2}. Then the optimal solution $(\rho,u,E)$ satisfies the exponential turnpike property
\begin{equation}\label{eq:4.4}
    \int_{\mathbb{R}^D} \rho(t,x) |u(t,x)-\bar{v}|^2 + 2\rho(t,x)e(t,x)dx  \leq C e^{-\alpha t} \int_{\mathbb{R}^D} \rho_0(x)|u_0(x)-\bar{v}|^2 + 2\rho_0(x)e_0(x)dx
\end{equation}
for any $t\in(0,T)$ when $T>0$ is sufficiently large.  Here the internal energy $e$ is given by $e=E-|u|^2/2$.
\end{itemize}
\end{theorem}

The subsequent three subsections are devoted to the proof of this theorem, addressing each level of the hierarchical framework.

\subsection{Proof of the theorem for the particle system}\label{Section4.1}

In order to obtain the exponential turnpike result, we firstly state the following lemma.
\begin{lemma}\label{lemma4.1}
    Denote the nonnegative cost function $\mathcal{E}=\mathcal{E}(t)$ such that 
    \begin{itemize}
        \item[(1)] There exists a constant $C_0>0$ such that  $\int_a^T\mathcal{E}(t)dt \leq C_0 \mathcal{E}(a)$ for any $0\leq a\leq T$.
        \item[(2)] There exists a constant $C_1>0$ such that  $\mathcal{E}(t_2) \leq C_1 \mathcal{E}(t_1)$ for any $0\leq t_1\leq t_2\leq T$.
    \end{itemize}
    Then there exist constants $C>0$ and $\alpha>0$, which are independent of $T$, such that:
    $$
    \mathcal{E}(t)  \leq C e^{-\alpha t} \mathcal{E}(0)
    $$
    for any $t\in(0,T)$ with $T$ sufficiently large.
\end{lemma}

This lemma is quite standard and we give a proof in Appendix \ref{AppendixA} for completeness.

\begin{remark}
The first assumption in Lemma \ref{lemma4.1} can be guaranteed by the cheap control property which we obtained in the previous section. Therefore, to prove the exponential turnpike property,  it suffices to prove that the optimal solutions satisfy the assumption (II) in Lemma \ref{lemma4.1}.
\end{remark}

For the optimal control problem \eqref{particle-eq}, \eqref{particle:cost}, we state
\begin{lemma}\label{lemma42}
    For the optimal solution $(x(t),v(t))$, we have 
    \begin{align*}
    &\|v(t_2)-\bar{v}\|_N^2 \leq C_{\lambda} \|v(t_1)-\bar{v}\|^2_N.
\end{align*}
Here the constant $C_{\lambda}>1$ satisfies 
$$
C_{\lambda} = \left\{
\begin{array}{ll}
   1 + \lambda^{-1/2},  & \qquad \lambda \leq 1,\\[4mm]
   1 + \lambda^{1/2},  & \qquad \lambda > 1.
\end{array}
\right.
$$
\end{lemma}
\begin{proof}
    Multiplying $(v_i-\bar{v})$ on the left of the second equation yields
    \begin{align*}
    \frac{d\|v-\bar{v}\|_N^2}{dt}
    =&~ \frac{2}{N^2}\sum_{i,j=1}^N \Psi(|x_i-x_j|) \langle {v}_i-\bar{v}, {v}_j-{v}_i\rangle +\frac{2}{N}\sum_{j=1}^N \langle {v}_i-\bar{v}, f_i \rangle \\
    \leq &~\frac{2}{N}\sum_{j=1}^N \langle {v}_i-\bar{v}, f_i \rangle \\[1mm]
    \leq &~ \|v-\bar{v}\|_N^2 +  \|f\|_N^2.
\end{align*}
When $\lambda\leq 1$, we integrate from $t_1$ to $t_2$ and get
\begin{align*}
    \|v(t_2)-\bar{v}\|_N^2 - \|v(t_1)-\bar{v}\|_N^2
    \leq&~ \frac{1}{\lambda} \int_{t_1}^{t_2} \|v(t)-\bar{v}\|_N^2 + \lambda \|f\|^2 dt \leq \frac{1}{\sqrt{\lambda}} \|v(t_1)-\bar{v}\|_N^2.
\end{align*}
Notice that the last inequality follows from the cheap control property. As $\lambda>1$, we get
\begin{align*}
    \|v(t_2)-\bar{v}\|_N^2 - \|v(t_1)-\bar{v}\|_N^2
    \leq& \int_{t_1}^{t_2} \|v(t)-\bar{v}\|_N^2 + \lambda \|f\|^2 dt \leq \sqrt{\lambda} \|v(t_1)-\bar{v}\|_N^2.
\end{align*}
This completes the proof.
\end{proof}

Based on this lemma and Lemma \ref{lemma4.1}, we can prove the exponential turnpike property for the particle system, which corresponds to case (1) in Theorem \ref{thm14} 

\subsection{Proof of the theorem for the mean-field problem}\label{Section4.2}

Since $\mu_0$ has bounded second-order moment, we can take the empirical measure at $t=0$ such that $\|v(0)-\bar{v}\|_N$ is uniformly bounded.
From the estimate \eqref{expo-particle}, we know that 
$$
\int_{\mathbb{R}^d} |v-\bar{v}|^2 ~d\mu_N(t,x,v) = \|v(t)-\bar{v}\|_N^2
$$
is uniformly bounded with respect to $N$. Then we recall Theorem \ref{thm21} and obtain the convergence in the Wasserstein distance $\mathcal{W}_2$. In this way, we take $N\rightarrow \infty$ in \eqref{expo-particle} and obtain the exponential turnpike estimate in the mean field level.

Alternatively, the result on the mean field problem can be also proven by a direct estimate of \eqref{mean-field-pb}. 
Namely, one may use a similar argument as that in Lemma \ref{lemma42} to prove
\begin{eqnarray*}
\begin{aligned}
    &\int |v-\bar{v}|^2 d\mu(t_2,x,v) \leq C_{\lambda} \int |v-\bar{v}|^2 d\mu(t_1,x,v),\qquad \forall~0\leq t_1 \leq t_2 \leq T.
\end{aligned}
\end{eqnarray*}
Here $\mu(t,x,v)$ represent the optimal solution to the problem \eqref{mean-field-pb}. 

\subsection{Proof of the theorem for the hydrodynamic equations}\label{section4.3}
Firstly, we consider the optimal control problem for the pressure-less Euler equation.
\begin{lemma}
For the optimal solution $(\rho,u)$ of the optimal control problem \eqref{hydro1-pb1}---\eqref{hydro1-pb4}, we have 
$$
\int_{\mathbb{R}^D} \rho(t_2,x) |u(t_2,x)-\bar{v}|^2 dx\leq C_\lambda \int_{\mathbb{R}^D} \rho(t_1,x) |u(t_1,x)-\bar{v}|^2 dx.
$$
Here $C_\lambda$ is a positive constant given in Lemma \ref{lemma42}.
\end{lemma}
\begin{proof}
    For the optimal solution $(\rho,u)$ with the optimal control $f_H$, we use Lemma \ref{lemma31} to obtain
    $$
    \int_{\mathbb{R}^D} \rho(t_2,x) |u(t_2,x)-\bar{v}|^2 dx \leq \int_{\mathbb{R}^D} \rho(t_1,x) |u(t_1,x)-\bar{v}|^2 dx + 2 \int_{t_1}^{t_2}\int_{\mathbb{R}^D} \rho (u-\bar{v})\cdot f_H dxdt.
    $$
The last term can be bounded by
    \begin{align*}
        2 \int_{t_1}^{t_2}\int_{\mathbb{R}^D} \rho (u-\bar{v})\cdot f_H~ dxdt &\leq 
        \left\{\begin{array}{lc}
        \lambda^{-1}\int_{t_1}^{t_2}\int_{\mathbb{R}^D} \rho |u-\bar{v}|^2 + \lambda \rho|f_H|^2 dxdt &\quad  \lambda\leq 1\\[5mm]
        \int_{t_1}^{t_2}\int_{\mathbb{R}^D} \rho |u-\bar{v}|^2 + \lambda \rho|f_H|^2 dxdt & \quad \lambda>1.
        \end{array}\right.
    \end{align*}
Moreover, thanks to the cheap control property, we obtain the estimate 
$$
\int_{\mathbb{R}^D} \rho(t_2,x) |u(t_2,x)-\bar{v}|^2 dx \leq C_\lambda \int_{\mathbb{R}^D} \rho(t_1,x) |u(t_1,x)-\bar{v}|^2 dx
$$
and complete the proof of lemma.
\end{proof}

Based on this lemma and Lemma \ref{lemma3.4}, we obtain the exponential turnpike property in \eqref{eq4.3} according to the result in Lemma \ref{lemma4.1}. This completes the proof of case (3a) in Theorem \ref{thm14}.

At last, we consider the optimal control problem for the full Euler equation and give
\begin{lemma}
For the optimal solution $(\rho,u,E)$ of the optimal control problem \eqref{hydro2-pb1}---\eqref{hydro2-pb5}, we have 
\begin{align*}
&\int_{\mathbb{R}^D} \rho(t_2,x) |u(t_2,x)-\bar{v}|^2 + 2\rho(t_2,x)e(t_2,x)dx \\[2mm]
\leq&~ C_\lambda \int_{\mathbb{R}^D} \rho(t_1,x) |u(t_1,x)-\bar{v}|^2 + 2\rho(t_1,x)e(t_1,x)dx.
\end{align*}
Here $C_\lambda$ is a positive constant given in Lemma \ref{lemma42}.
\end{lemma}
\begin{proof}
    For the optimal solution $(\rho,u,E)$ with the optimal control $F_1$ and $F_2$, we use Lemma \ref{lemma3.5} to obtain
    \begin{align*}
    \int_{\mathbb{R}^D} \rho|u-\bar{v}|^2 + 2 \rho e ~dx\Big|_{t=t_1}^{t=t_2} \leq 2\int_{t_1}^{t_2}\int_{\mathbb{R}^D}\rho (F_2 - \bar{v}\cdot F_1)~dxdt.
\end{align*}
The right hand side can be bounded by
\begin{align*}
&~2 \int_{t_1}^{t_2}\int_{\mathbb{R}^D} \rho (F_2 - \bar{v}\cdot F_1)~ dxdt = 
2 \int_{t_1}^{t_2}\int_{\mathbb{R}^D} \rho \Big(F_2 - u\cdot F_1 + (u-\bar{v})\cdot F_1\Big)~ dxdt\\[2mm]
\leq &~
\int_{t_1}^{t_2}\int_{\mathbb{R}^D} 2\rho e + \frac{\rho}{2e}|F_2 - u\cdot F_1|^2 + \rho|u-\bar{v}|^2 + \rho |F_1|^2~ dxdt.
\end{align*}
Then we have for $\lambda>1$
\begin{align*}
2 \int_{t_1}^{t_2}\int_{\mathbb{R}^D} \rho (F_2 - \bar{v}\cdot F_1)~ dxdt \leq \int_{t_1}^{t_2}\int_{\mathbb{R}^D} 2\rho e +  \rho|u-\bar{v}|^2 + \lambda G(F_1,F_2)~ dxdt
\end{align*}
and for $\lambda\leq 1$
\begin{align*}
2 \int_{t_1}^{t_2}\int_{\mathbb{R}^D} \rho (F_2 - \bar{v}\cdot F_1)~ dxdt \leq \lambda^{-1}\int_{t_1}^{t_2}\int_{\mathbb{R}^D} 2\rho e +  \rho|u-\bar{v}|^2 + \lambda G(F_1,F_2)~ dxdt.
\end{align*}
We use the cheap control property in Lemma \ref{lemma3.6} to obtain the estimate stated in the lemma.
\end{proof}

Consequently, we combine this lemma with Lemma \ref{lemma3.6} and Lemma \ref{lemma4.1} to conclude the exponential turnpike property in \eqref{eq:4.4} and complete the proof of case (3b) in Theorem \ref{thm14}.

\section{Numerical results of the feedback control}\label{Section5}

In the proof of the cheap control property, we design the feedback control for each problem separately. In applications, the proper feedback law itself is of wide interest. Therefore, in this subsection, we show some numerical results for the particle system and the hydrodynamic equations under the feedback control.

We firstly show the results for the particle system. Specifically, we compute the equation \eqref{particle-eq} with $N=30$ and the interaction kernel
$$
\Psi(x,y)=\frac{1}{(1 + |x - y|^2)^\gamma},\qquad \gamma=1.
$$
Besides, the feedback control is given by $f_i=-\beta(v-\bar{v})$ with $\beta=2$ and $\bar{v}=0.5$. Initial data for $x_0$ and $v_0$ are randomly generated from the normal distributions centered at $\nu_x=1$ and $\nu_v=0$.
\begin{figure}[h]
    \centering
    \begin{subfigure}{}
        \includegraphics[width=1.9in]{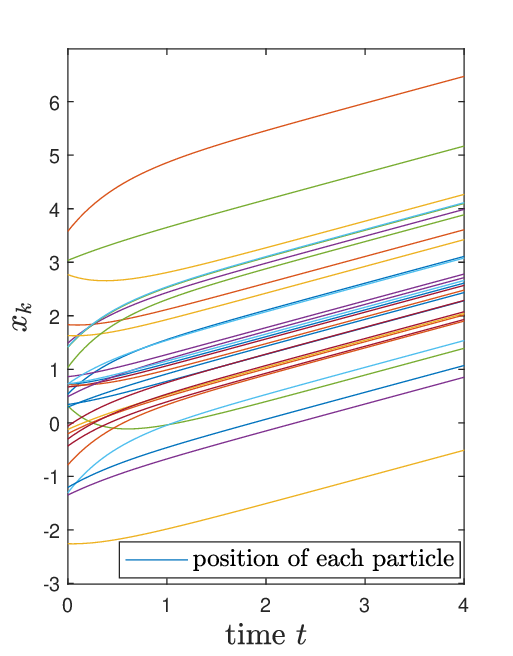}
    \end{subfigure}
    \begin{subfigure}{}
        \includegraphics[width=1.9in]{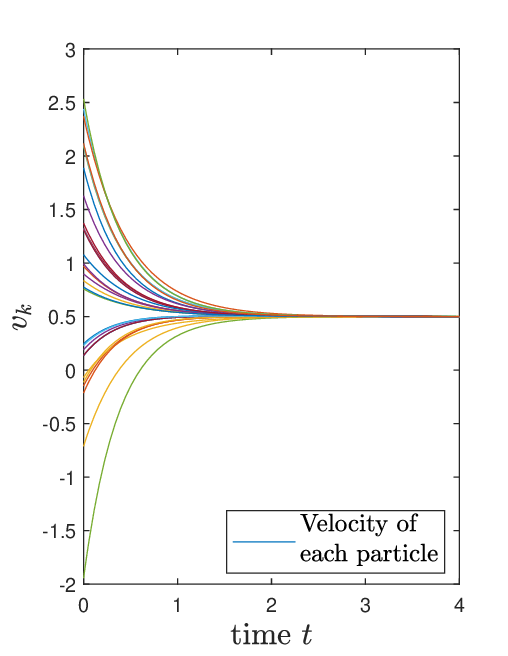}
    \end{subfigure}
    \begin{subfigure}{}
        \includegraphics[width=1.95in]{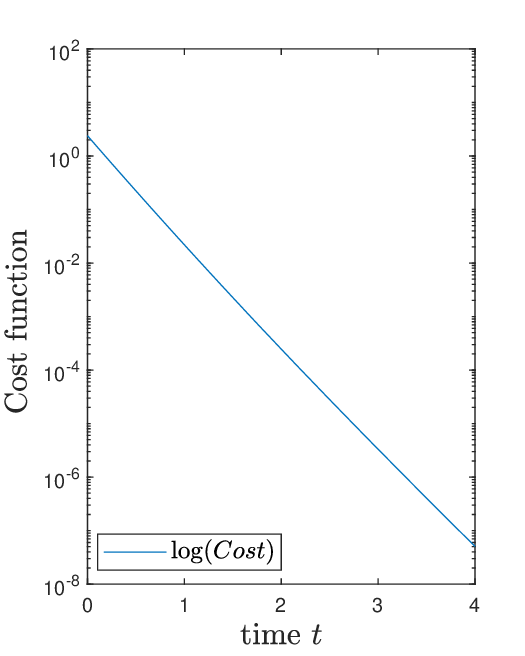}
    \end{subfigure}
    \caption{Numerical results for the particle system: (left) position $x_k$ of each particle; (middle) velocity $v_k$ of each particle; (right) the decay of the cost, represented in the logarithmic regime.}
    \label{fig1}
\end{figure}

In Figure \ref{fig1}, we plot the time evolutions of position $x_i=x_i(t)$ and velocity $v_i=v_i(t)$ for each particle.
Moreover, we plot the evolution of the cost $\|v(t)-\bar{v}\|^2_N + \lambda \|f_i(t)\|^2_N$, which is represented in the logarithmic regime. Here $\lambda=0.25$. From these figures, it is easy to observe that each velocity $v_k(t)$ tends to the goal $\bar{v}=0.5$ as time evolves. Besides, the cost $\|v(t)-\bar{v}\|^2_N + \lambda \|f_i(t)\|^2_N$ decays exponentially fast by using the feedback control.

Next we numerically compute the hydrodynamic equations with the feedback control. For the pressure-less case, we consider the equations \eqref{hydro1-pb2}---\eqref{hydro1-pb4} with the initial data 
$$
\rho(0,x) = 0.1,\qquad u(0,x)= \exp\left(-2{x^2}\right).
$$
The interaction kernel $\Psi(x,y)$ is taken to be the same as in the previous example. The feedback control is taken as 
$$
f_H(t,x)=-\beta\rho(t,x)(u(t,x)-\bar{v})
$$ 
with $\beta=2$ and $\bar{v}=-0.1$. In Figure \ref{fig2}, we show the solutions at $t=0$, $t=0.5$, and $t=4$ from left to right. From these, we observe that the velocity $u(t,x)$ tends to the constant $\bar{v}=-0.1$ for all $x$ as time evolves. Moreover, we plot in Figure \ref{fig2-1} the evolution of the cost $\rho|u-\bar{v}|^2+\gamma|f_H|^2$ with feedback law, which is represented in the logarithmic regime. Here $\lambda=1$. We observe that the cost decays exponentially fast by using the feedback control.
\begin{figure}[H]
    \centering
    \begin{subfigure}{}
        \includegraphics[width=1.9in]{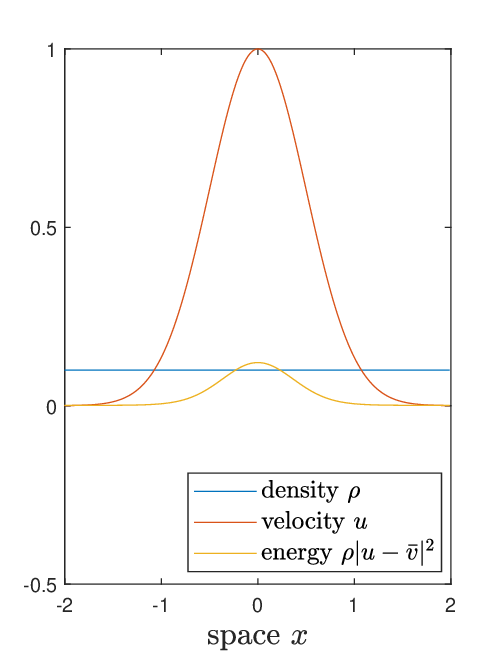}
    \end{subfigure}
    \begin{subfigure}{}
        \includegraphics[width=1.9in]{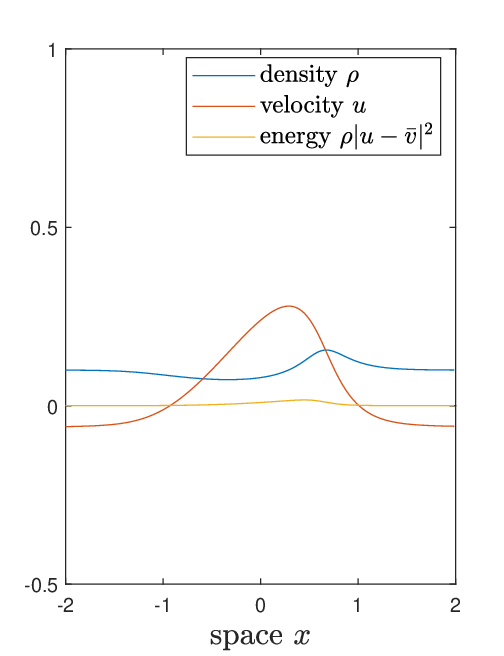}
    \end{subfigure}
    \begin{subfigure}{}
        \includegraphics[width=1.9in]{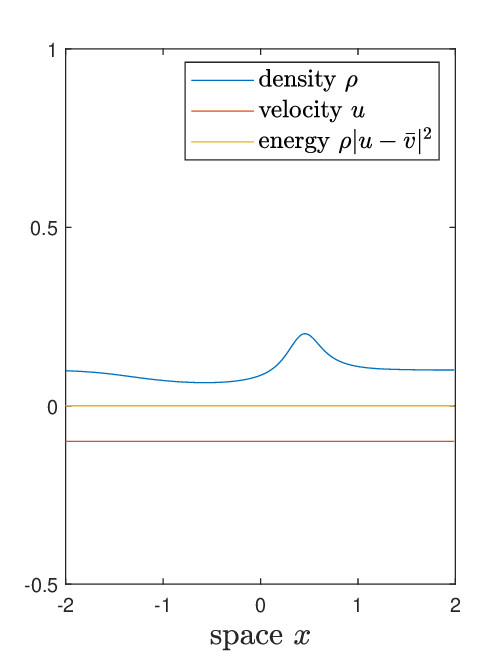}
    \end{subfigure}
    \caption{Numerical results for the pressure-less hydrodynamic equation: (left) the initial data at $t=0$; (middle) the solution at $t=0.5$; (right) the solution at $t=4$.}
    \label{fig2}
\end{figure}

\begin{figure}[H]
    \centering
    \includegraphics[width=2.5in]{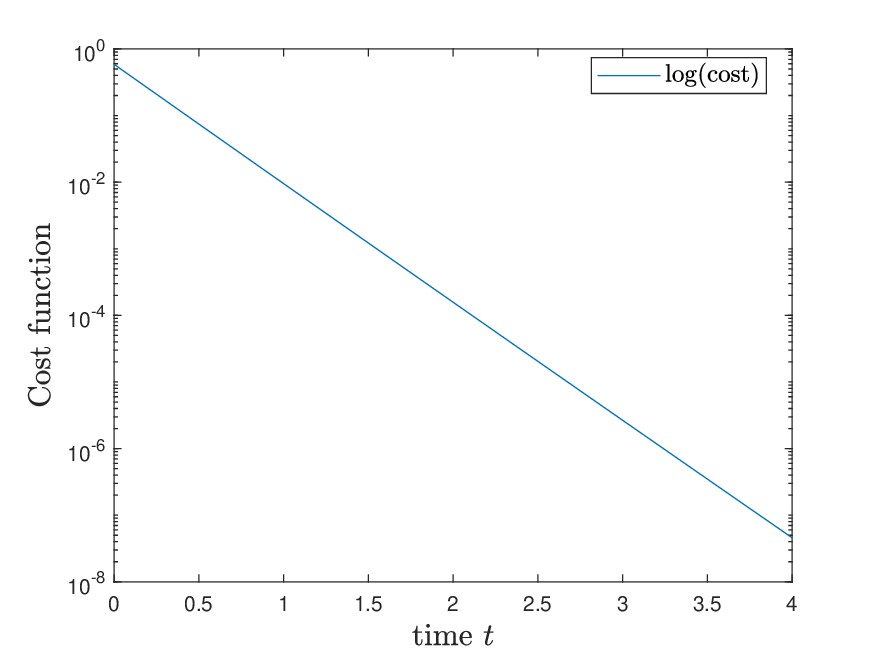}
    \caption{The pressure-less hydrodynamic equation: the decay of the cost function $\rho|u-\bar{v}|^2+\gamma|f_H|^2$ with feedback law, represented in the logarithmic regime.}
    \label{fig2-1}
\end{figure}

The last example of this subsection is the computation of the full Euler equation with the feedback control. We numerically compute the hydrodynamic equations \eqref{eq3:rho}---\eqref{hydro2-pb5} with the initial data 
$$
\rho(0,x) = 0.1,\qquad u(0,x)= \exp\left(-2{x^2}\right),\qquad p(0,x)=0.01.
$$
The interaction kernel $\Psi(x,y)$ is given to be the same as in the first example. The feedback control is taken as 
$$
{F}_1=-\beta({u}-\bar{v}),\qquad {F}_2=-\beta[2{e}+{u}\cdot({u}-\bar{v})]
$$ 
with $\beta=2$ and $\bar{v}=0.1$. In Figure \ref{fig2}, we show the solutions at $t=0$, $t=0.5$, and $t=4$ from left to right. These results indicate that the velocity $u(t,x)$ converges to the constant $\bar{v}=0.1$ for all $x$ as time evolves. Similarly, we plot the evolution of the cost $\rho|u-\bar{v}|^2+2\rho e+\gamma G(F_1,F_2)$ with the feedback law and $\lambda=1$. Clearly, the cost also decays exponentially fast with this feedback control.
\begin{figure}[H]
    \centering
    \begin{subfigure}{}
        \includegraphics[width=1.9in]{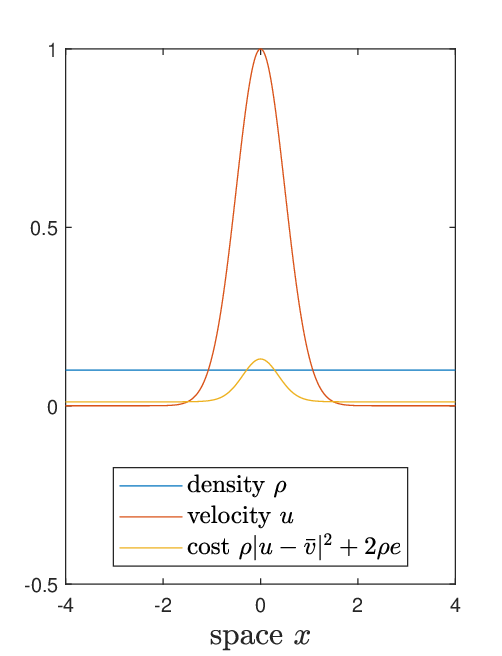}
    \end{subfigure}
    \begin{subfigure}{}
        \includegraphics[width=1.9in]{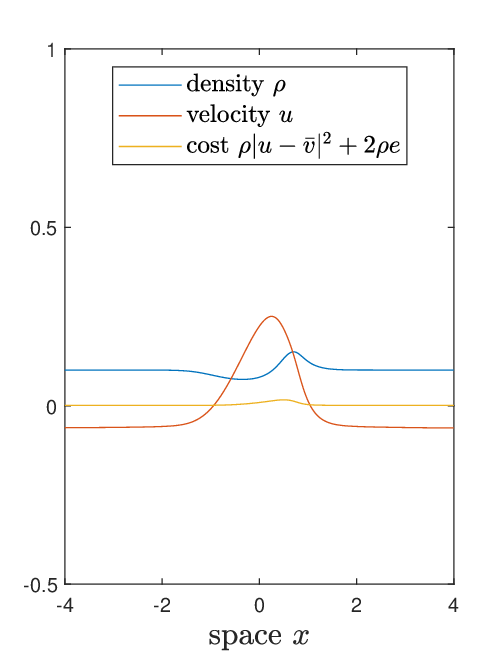}
    \end{subfigure}
    \begin{subfigure}{}
        \includegraphics[width=1.9in]{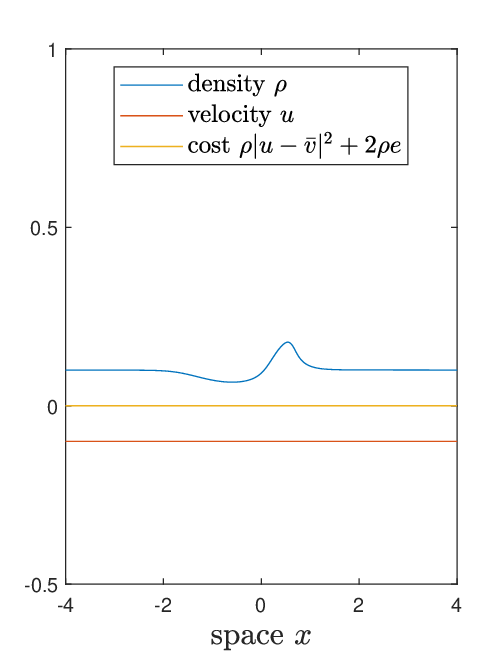}
    \end{subfigure}
    \caption{Numerical results for the Euler equation: (left) the initial data at $t=0$; (middle) the solution at $t=0.5$; (right) the solution at $t=4$.}
    \label{fig3}
\end{figure}

\begin{figure}[H]
    \centering
    \includegraphics[width=2.5in]{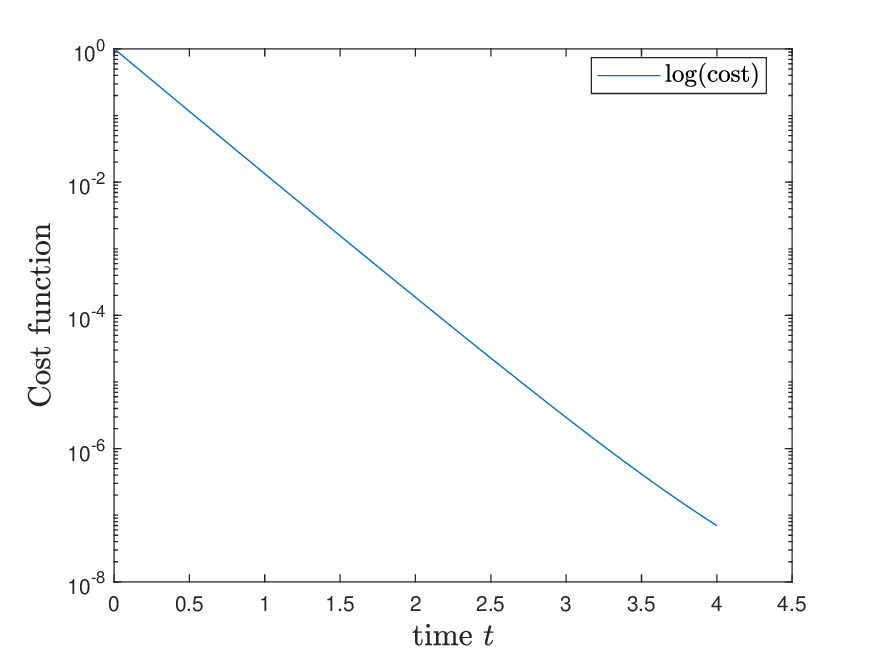}
    \caption{The Euler equation: the decay of the cost function $\rho|u-\bar{v}|^2+2\rho e+\gamma G(F_1,F_2)$ with feedback law, represented in the logarithmic regime.}
    \label{fig3-1}
\end{figure}

\section{Conclusion}

Starting from the optimal control problems of particle systems, we formally formulate optimal control problems at the kinetic and hydrodynamic level. Assuming the existence of solutions with certain regularity, we establish the exponential turnpike property for this hierarchy of optimal control problems, showing that the optimal trajectories remain exponentially close to the corresponding steady states over long time horizons.
This work extends our previous result \cite{Herty_Zhou_2025} to the framework of second-order particle systems and, for the first time, establishes the exponential turnpike property for hydrodynamic models. In particular, the results confirm that the exponential turnpike behavior is preserved throughout the entire hierarchy of optimal control problems, consistently bridging the microscopic, kinetic, and macroscopic levels.

\section*{\normalsize{Acknowledgments}} 
The authors thank the Deutsche Forschungsgemeinschaft (DFG, German Research Foundation) for the financial support through HE5386/33-1 Control of Interacting Particle Systems, and Their Mean-Field, and Fluid-Dynamic Limits (560288187) and through 442047500/SFB1481 within the projects B04 (Sparsity fördernde Muster in kinetischen Hierarchien), B05 (Sparsifizierung zeitabhängiger Netzwerkflußprobleme mittels diskreter Optimierung) and B06 (Kinetische Theorie trifft algebraische Systemtheorie). The second author is funded by Alexander von Humboldt Foundation (Humboldt Research Fellowship Programme for Postdocs).

\section*{\normalsize{Competing interests}} 
The authors declare none.

\begin{appendices}
    \section{Missing proof}\label{AppendixA}
    \begin{proof}[Proof of Lemma \ref{lemma4.1}]
For a fixed constant $\tau$, there exists a point $t_1\in[0,\tau]$ such that
\begin{align*}
    \mathcal{E}(t_1) \leq&~ \frac{1}{\tau} \int_{0}^{\tau}
    \mathcal{E}(t) dt 
    \leq \frac{C_0}{\tau}\mathcal{E}(0).
\end{align*}
For any $t\geq \tau \geq t_1$, we obtain 
\begin{align*}
    \mathcal{E}(t)
    \leq C_1\mathcal{E}(t_1)
    \leq \frac{C_1C_0}{\tau} \mathcal{E}(0).
\end{align*}
Then we prove the following inequality for $t \in [n\tau,T]$ by induction:
\begin{align}\label{proof-step1}
\mathcal{E}(t) \leq \left(\frac{C_0C_1}{\tau}\right)^n \mathcal{E}(0).
\end{align} 
Suppose the inequality holds for $n\geq 1$. There exists $t_n\in[n\tau,(n+1)\tau]$ such that
\begin{align*}
    \mathcal{E}(t_n) \leq&~ \frac{1}{\tau} \int_{n\tau}^{(n+1)\tau} \mathcal{E}(t) dt    \leq \frac{C_0}{\tau} \mathcal{E}(n\tau) \leq \frac{C_0}{\tau} \left(\frac{C_0C_1}{\tau}\right)^n \mathcal{E}(0) .
    \end{align*}
Thus for any $t\in[(n+1)\tau,T]$, we obtain 
$$
\mathcal{E}(t) \leq C_1 \mathcal{E}(t_n) \leq \left(\frac{C_0C_1}{\tau}\right)^{n+1} \mathcal{E}(0)
$$ 
and this completes the proof of \eqref{proof-step1}.
 
 Now we fix the constant $\tau$ in \eqref{proof-step1} such that $\tau>C_0C_1$. Next, we discuss the cases $t\in(0,\tau)$ and $t\in[\tau,T)$ separately. 
    
    For any $t\in[\tau,T)$, we take the integer $n=\lfloor t/\tau \rfloor$. Then,  $1\leq n \leq \frac{T}{\tau}$  and $t\in [n\tau,T)$
    and we obtain by \eqref{proof-step1}:
\begin{align*}
    \mathcal{E}(t) \leq &~ \left(\frac{C_0C_1}{\tau}\right)^n \mathcal{E}(0).
\end{align*}
Due to the definition of $n$, we have  $n>t/\tau-1$. Also, the constant $\tau$ is chosen such that $\tau>C_0C_1$. Thus we have
\begin{align*}
    &~ \left(\frac{C_0C_1}{\tau}\right)^n = \left(\frac{\tau}{C_0C_1}\right)^{-n} 
    \leq \left(\frac{\tau}{C_0C_1}\right)^{1-t/\tau}.
\end{align*}
The exponential estimate is then given by
$$
    \mathcal{E}(t) \leq C_2 e^{-\alpha t} \mathcal{E}(0),\quad \forall ~t\in[\tau, T)
$$
with 
$$
C_2 = \frac{\tau}{C_0C_1},\qquad \alpha = \frac{1}{\tau}\log\left(\frac{\tau}{C_0C_1}\right)>0.
$$

On the other hand, for $t\in(0,\tau)$, we have
$$
C_2 e^{-\alpha t} \geq  C_2 e^{-\alpha \tau} = 1.
$$
By assumption (2), we have 
\begin{equation*}
\mathcal{E}(t) \leq C_1
 C_2 e^{-\alpha t} \mathcal{E}(0).
\end{equation*}
To combine the results of $t\in(0,\tau)$ and $t\in[\tau, T)$, we  take $C=C_1C_2$ and obtain
\begin{equation}\label{thm14-eq4}
\mathcal{E}(t)
\leq  
C e^{-\alpha t} \mathcal{E}(0),\quad \forall~t\in(0,T).
\end{equation}
This completes the proof of lemma.
\end{proof}
\end{appendices}


\end{document}